\documentclass[10pt,a4paper,article]{amsart}
\usepackage{amssymb,stmaryrd}
\usepackage{amsthm,amstext}
\usepackage{amssymb,amsmath}
\usepackage{amsfonts}
\usepackage{MnSymbol,slashed}
\usepackage{fancyhdr}
\usepackage{graphicx}
\usepackage{pdfpages}
\usepackage{float}
\usepackage{subfig}
\usepackage{simplewick}
\usepackage{graphics,psfrag}
\usepackage{pstricks}
\usepackage{latexsym}
\usepackage{changepage}
\usepackage{caption}
\usepackage{paralist}

\theoremstyle{plain}
\newtheorem{theorem}{Theorem}[section] 

\newtheorem{lemma}[theorem]{Lemma}

\newtheorem{proposition}[theorem]{Proposition}
\def\k{{{\mathbb{K}}}}
\newcommand{\fun}[3]{#1: #2 \rightarrow #3}
\newcommand{\onef}[1]{\Omega^{#1}}
\newcommand{\tchrs}[2]{{\Gamma}^{#1}_{\hphantom{#1}#2}}

\newcommand{\dirac}{{ \slashed{D} }}

\newcommand{\CS}{{\hbox{{$\mathcal S$}}}}

\newcommand{\CP}{\hbox{{$\mathcal P$}}}

\newcommand{\CJ}{\hbox{{$\mathcal J$}}}

\newcommand{\C}{\mathbb{C}}
\newcommand{\R}{\mathbb{R}}

\newcommand{\Z}{\mathbb{Z}}

\newcommand{\cg}{\hbox{{$\mathfrak g$}}}

\newcommand{\del}{\partial}

\newcommand{\extd}{\mathrm{d}}

\newcommand{\eps}{{\epsilon}}
\newcommand{\tens}{\mathop{{\otimes}}}
\newcommand{\la}{{\triangleright}}

\newcommand{\id}{\mathrm{id}}

\newcommand{\<}{\langle}
\renewcommand{\>}{\rangle}

\def\rcross{{\triangleright\!\!\!<}}

\begin{document}

\author{Evelyn Lira Torres and Shahn Majid}
\address{School of Mathematical Sciences\\ Queen Mary University of London \\ Mile End Rd, London E1 4NS }
\email{ e.y.liratorres@qmul.ac.uk, s.majid@qmul.ac.uk}

\title{Quantum Gravity and Riemannian Geometry on the Fuzzy Sphere}
	\begin{abstract} We study the quantum geometry of the fuzzy sphere defined as the angular momentum algebra $[x_i,x_j]=2\imath\lambda_p \epsilon_{ijk}x_k$ modulo setting $\sum_i x_i^2$ to a constant, using a recently introduced 3D rotationally invariant differential structure. Metrics are given by symmetric $3 \times 3$ matrices $g$ and we show that for each metric there is a unique quantum Levi-Civita connection with constant coefficients, with scalar curvature $ \frac{1}{2}({\rm Tr}(g^2)-\frac{1}{2}{\rm Tr}(g)^2)/\det(g)$. As an application, we construct  Euclidean quantum gravity on the fuzzy unit sphere. We also calculate the charge 1 monopole for the 3D differential structure.  
	\end{abstract}
\maketitle

\section{Introduction}

The angular momentum algebra $U(su_2)$ has been viewed since the 1970s as the quantisation of $\R^3$ viewed as $su_2^*$ with its Kirillov-Kostant bracket as part of a general theory for any Lie algebra. As such, setting the quadratic Casimir to a constant quantises the coadjoint orbits, again in a standard way. The angular momentum algebra was also proposed  as `position coordinates' for  Euclideanised 2+1 quantum gravity by 't Hooft\cite{Hoo}. We denote it $\C_\lambda[\R^3]$ or `fuzzy $\R^3$ with generators $x_i$ and relations $[x_i,x_j]=2\lambda \epsilon_{ijk}x_k$ as in \cite{BatMa,BegMa,Ma:spo,FreMa} to indicate that we consider it a deformation of flat spacetime. Its covariance under the quantum double $D(U(su_2))=U(su_2)\rcross \C[SU_2]$ as `Poincare group' was found in \cite{BatMa} along with a 4D quantum-Poincar\'e invariant calculus, and further studied in \cite{Ma:spo}\cite{FreMa} among other places. This is by now a  well-established picture of a deformed $\R^3$  in Euclideanised 2+1 quantum gravity with point sources and without cosmological constant, see e.g.  \cite{FreLiv} at the group algebra level. Moreover, it deforms naturally to the quantum enveloping algebra $U_q(su_2)$ with quantum Poincar\'e group $D(U_q(su_2))\cong \C_q[SO_{1,3}]$ in 2+1 quantum gravity with cosmological constant, see \cite{MaSch} for an overview and the relationship to the bicrossproduct model Majid-Ruegg quantum spacetime\cite{MaRue}.  

In physical terms, $\lambda=\imath\lambda_p$ where $\lambda_p$ is a real deformation parameter which, in the above context (but not necessarily), should be  of order the Planck scale. The general idea that spacetime geometry is `quantum' or noncommutative was speculated since the 1920s but in modern times was proposed in \cite{Ma:pla} coming out of ideas for quantum gravity of quantum Born reciprocity or observable-state/Hopf algebra duality.  See also subsequent works by many authors, including \cite{AmeMa}. Not surprisingly, however,  the quantum spacetime $\C_\lambda[\R^3]$ with its natural rotationally invariant quantum metric is flat and admits only the zero Levi-Civita connection in the standard coordinates \cite[Example 8.15]{BegMa}. The same is true for the corresponding bicrossproduct model spacetime with quantum Poincar\'e covariant calculus\cite[Prop.~9.20]{BegMa}, the two models being related by twisting\cite{MaOse}. For quantum gravity effects with background gravity present we should look at curved quantum spacetimes.  

Indeed, one might expect such quantum Riemannian geometry to be more interesting for the corresponding `fuzzy sphere' quotient, but it turns out that the differential structure proposed in \cite{BatMa} does not descend to the fuzzy sphere. To address this problem, \cite[Example 1.46]{BegMa} recently proposed a different 3D differential structure on the fuzzy sphere and in the present paper we explore its quantum Riemannian geometry with this calculus.  We find that it is indeed curved for general metrics, including its natural rotationally invariant `round metric'. We will denote the unit fuzzy sphere here by $\C_\lambda[S^2]$. The term `fuzzy sphere' is also used in the literature, e.g. \cite{Mad}, for matrix algebras $M_n(\C)$ viewed in our terms as further quotients of $\C_\lambda[S^2]$ for certain values of $\lambda$ (those values that descend to the irreducible $n$-dimensional representations of spin $n/2$). 

We use the constructive `quantum groups' approach to quantum Riemannian geometry as in the recent text \cite{BegMa}. This was established in recent years e.g. \cite{BegMa1,BegMa2, BegMa5, Ma:gra, Ma:squ, Ma:haw, MaPac2,ArgMa} using particularly (but not only)  the notion of a bimodule connection\cite{DVM,Mou}. The formalism is recalled briefly in Section~\ref{secpre} along with the new differential structure on $\C_\lambda[S^2]$ proposed in \cite{BegMa}. Section~\ref{secqlc} contains the first new results, namely uniqueness and construction of a quantum-Levi Civita connection for each metric. Metrics here can be chosen freely as symmetric $3\times 3$ matrices in the natural basis of the 3D calculus. In Section~\ref{seccurv} we look at the curvature as a function of the metric and use this in Section~\ref{secqg} to explore Euclidean quantum gravity on the fuzzy sphere. Conventionally, the  Euclidean case, although not usual quantum gravity itself, is nevertheless of interest on any compact Riemannian manifold with boundary\cite{HawGib}. 

Note also that our approach to quantum Riemannian geometry is very different from that of Connes\cite{Con} based on spectral triples as abstract `Dirac' operators, but the two approaches can sometimes be related \cite{BegMa3}. A first step for the fuzzy sphere would be to compute the Grassmann or monopole connection, which we do in the short Section~\ref{secmon}. 

It remains to explain why the cotangent bundle on the fuzzy sphere in this paper is 3D not 2D. Indeed, it often happens in quantum geometry that there is an obstruction to having the same dimension connected differential calculus as classically and preserving symmetries. This is because most highly noncommutative geometries are inner in the sense of a 1-form $\theta$ such that the exterior derivative is $\extd=[\theta,\ \}$ but this equation has no meaning in classical geometry (the right hand side would be zero) and indeed it is a purely quantum phenomenon. Hence quantum calculi often contain an extra dimension not visible classically, which we could think of as an internally generated `time' direction $\theta=\extd t$, since quite often the partial derivative in the $\theta$ direction turns out to be the natural wave operator or Laplacian. This was explained in \cite{Ma:spo} in the context of the 4D calculus on $\C_\lambda[\R^3]$, where this external time is not part of spacetime but may be more related to geodesic flow proper time according to recent ideas in \cite{BegMa5} or to renormalisation group flow according to ideas in \cite{FreMa}. For our new calculus on $\C_\lambda[S^2]$, we actually do have a (nonconnected) 3D calculus on $\C_\lambda[\R^3]$ with no extra dimension there, but when this descends to the sphere quotient, the  special inner element $\theta=(2\imath\lambda_p)^{-2}x_i\extd x_i$, i.e. which geometrically would be the normal to the sphere, does not decouple as it would classically and make the calculus on the fuzzy sphere 3D not 2D. 

\section{Preliminaries} \label{secpre}

Here we give a very short introduction to the general formalism, with more details to be found in \cite{BegMa} and references there in. This provides the framework whereby our constructions for the fuzzy sphere should not be seen as ad-hoc but natural within this context. To this end, let $A$ be a unital algebra, possibly noncommutative, over a field $\k$ (we will be mainly interested in $\C$). By a  first order \textit{differential calculus} $(\Omega^1,\extd)$ we mean that:

\begin{enumerate}
	\item[(1)] $\onef{1}$ is an $A$-bimodule;
	\item[(2)] A linear map $\fun{\extd }{A}{\onef{1}}$ such $\extd (ab)=(\extd a)b+a\extd b$  for all $a, b \in A$;
		\item[(3)] $\onef{1}={\rm span} \{ a \extd  b \, | \, a,b \in A \}$;
	\item[(4)] \textit{(Optional)} ${\rm ker} (\extd )=\k .1_A$
	\end{enumerate}

Here $\extd$ is called the \textit{exterior derivative}, the condition $(3)$ is the surjectivity condition and $(4)$ is the \textit{connectedeness property},  which
is not an axiom but is desirable. We also require $\Omega^1$ to extend to an exterior algebra $\Omega=\oplus_i\Omega^i$ of forms of different degree, generated by $A=\Omega^0$ and $\Omega^1$ with $\extd$ extending by the graded Leibniz rule and $\extd^2=0$. The product of $\Omega$ is denoted $\wedge$. For quantum Riemannian geometry we only need up to $\Omega^2$. See \cite[Chap. 1]{BegMa}. 

In this context,  we define a quantum metric as  $g\in \onef{1}\otimes_A \onef{1}$ such that there exists an inverse $(\ ,\ ):\Omega^1\tens\Omega^1\to A$ which is a bimodule map. Inverse here means in the usual sense but turns out to require that $g$ is central. We usually (but not always) also require $g$ to be quantum symmetric in the sense $\wedge(g)=0$. If this does not hold, we speak of an asymmetric or `generalised' metric.

Next, we left connection on $\Omega^1$ for us means a linear map $\nabla:\Omega^1\to \Omega^1\tens\Omega^1$ obeying the left Leibniz rule
\[ \nabla(a.\omega)=\extd a\tens\omega+a.\nabla\omega\]
for all $a\in A$ and $\omega\in \Omega^1$. If $X:\Omega^1\to A$ is a right vector field in the sense that it commutes with the action of $A$ from the right, then we may define $\nabla_X=\cdot(X\tens\id)\nabla:\Omega^1\to \Omega^1$ which then behaves like a usual covariant derivative. In the classical case with local coordinates $x^i$, say, we would set $\nabla\extd x^i=-\Gamma^i{}_{jk}\extd x^j\tens\extd x^k$ in terms of Christoffel symbols. 
One can apply a similar definition for any vector bundle in the sense of a left $A$-module $E$ (typically required to be projective). In our case $E=\Omega^1$ is a bimodule and we demand 
\[ \nabla (\eta a)=(\nabla \eta)a+\sigma (\eta \otimes\extd  a),\] 
for a bimodule map $\fun{\sigma}{\onef{1}\otimes _A \onef{1}}{\onef{1}\otimes _A \onef{1}}$, called the \textit{generalised braiding}. If this exists, it is uniquely determined by this formula, so this is not additonal data, just a property of some left connections. We say that $\nabla$ is then a {\em bimodule connection}\cite{DVM,Mou}. This case is nice because bimodule connections can be tensor producted. Relevant to us is that $\Omega^1\tens_A\Omega^1$ gets a bimodule connection
\[ \nabla(\omega\tens\eta)=\nabla\omega\tens\eta+ (\sigma(\omega\tens( ))\tens\id)\Delta\eta\]
for $\omega,\eta\in \Omega^1$. In this case $\nabla g=0$ makes sense and we say when that holds that $g$ is {\em metric compatible}\cite{BegMa2}. Explicitly, we need 
\[ \nabla g := (\nabla \otimes \text{id})g+(\sigma \otimes \text{id})(\text{id}\otimes\nabla )g=0. \]
See \cite[Chap 8]{BegMa} for more details. 

Also, for any left connection on $\Omega^1$ and a choice of $\Omega^2$, we have a canonical notion of torsion, which is the standard notion but written in terms of differential forms as
 \[T_{\nabla}=\wedge\nabla-\extd: \Omega^1\to \Omega^2\]
 Given a generalised metric we also have a notion of `cotorsion' define by \[coT_{\nabla}=dg^1 \otimes g^2 - g^1 \wedge \nabla g ^2 \in \onef{2}\otimes _{A}\onef{1}.\]
 which classically is a skew-symmetrized version of metric compatibility. We say that a connection is a {\em quantum Levi-Civita connection} (QLC) for a metric $g$  if it is torsion free and metric compatible. We say that it is a weak QLC (WQLC) if it is cotorsion free and torsion free. One can show that a QLC is necessarily a WQLC so it can be useful to impose the WQLC condition first, being linear in $\nabla$ compared to the QLC condition which is quadratic (due to the $\sigma$, which is linear in $\nabla$) and hence much harder to solve. 
 
Finally, over $\C$, we need everything to be `unitary' or `real' in a suitable sense. It means that $A$ is a $*$-algebra in the usual sense, and that $*$ extends to $\Omega$ in a way that commutes with $\extd$ and is a graded-order reversing involution (it means there is an extra minus sign on a product of odd degree forms). We require the metric to and connection to be `real' in the sense
\[  {\rm flip}(*\tens *)g=g,\quad  \sigma\circ {\rm flip}(*\tens *)\nabla= \nabla \circ *\]
In the classical case with self-adjoint local coordinates, this would ensure that the metric and connection coefficients are real. These are a well-studied set of axioms for which many interesting examples are known, eg \cite{BegMa2, Ma:squ, Ma:haw,MaPac2}.

\subsection{Fuzzy sphere and its 3D differential calculus}

We work over $\C$ and start with the enveloping algebra $U(su_2)$ of the angular momentum Lie algebra, with basis $x_i$ normalised so that 
\[ [x_i,x_j]=2\imath\lambda_p \epsilon_{ijk}x_k\]
for a parameter $\lambda_p$. We call this $\C_\lambda[\R^3]$ as a quantisation of functions on $\R^3$. We take it as a $*$ algebra with $x_i^*=x_i$ and $\lambda_p$ real. Note that this has finite-dimensional irreducible representations $\rho_j$ labelled by a non-negative  half-integer $j$ of dimension $n=2j+1$ and in which $\sum_i x_i^2= (n^2-1)\lambda_p^2$ in our normalisation. We define the {\em unit fuzzy sphere} $A=\C_\lambda[S^2]$ as the quotient $U(su_2)$ modulo the relation
\[ \sum_i x_i^2= 1-\lambda_p^2\]
which we see descends to the spin $j$ representation precisely when $\lambda_p= 1/n$. We keep $\lambda_p$ as a free parameter, however.  Note that in all cases $A=\C_\lambda[S^2]$ is infinite-dimensional and therefore never a matrix algebra. 

Next we define $\Omega(\C_\lambda[\R^3])$ as a free 3D calculus with central basis $s^i$, $i=1,2,3$. This means we impose $[s^i , x_j ]=0$, and
we define a differential 
\[  \extd x_i= \epsilon_{ijk}x_{j}s^k\]
which one can check is translation and rotation invariant calculus, but not connected. Indeed, $\extd \sum_i x_i^2=0$ so there are different connected components according to any constant value of $ \sum_i x_i^2$. The calculus is inner with 
\[ \theta= \tfrac{1}{2\imath \lambda_p}x_i s^i=\tfrac{1}{(2\imath\lambda_p)^2}x_i\extd x_i=-\tfrac{1}{(2\imath\lambda_p)^2}(\extd x_i) x_i.\]
Finally, for the exterior algebra we take $s^i$ to be Grassmann, with
\[  s^i \wedge s^j + s^j \wedge s^i =0,\quad \extd s^i =-\tfrac{1}{2}\epsilon_{ijk}s^j \wedge s^k\]
The reader should be warned, however, that this is no longer inner in higher degree by $\theta$. That in turn means it is not the maximal prolongation
of the first order calculus, but is a natural quotient.

We then take the same form of calculus and $\extd$ for $A=\C_\lambda[S^2]$, where we add the unit sphere relation. This is compatible with $\extd$ for reasons already given and this time we obtain a connected calculus. These facts are all covered in \cite[Example~1.46]{BegMa} and one could do the same for a sphere of any fixed radius. 

\begin{lemma}\label{sl} In $\Omega(\C_\lambda[S^2])$, one has
\[  s^l= \tfrac{1}{(1-\lambda_p^2)}\left(\tfrac{1}{2\imath\lambda_p} x_lx_i \extd x_i +\epsilon_{lim}(\extd x_i) x_m\right)\]
\end{lemma}
\proof First observe from the form of $\extd x_i$ that $\epsilon_{lmi}\extd x_i   =\epsilon_{lmi}\epsilon_{ijk}x_j s^k =x_l s^m - x_m s^l$. Then 		\begin{align*} \epsilon_{lmi}\extd x_i x_m  &=(x_l s^m - x_m s^l)x_m  =x_l x_m s^m - x_m x_m s^l =
		 2i\lambda _p x_l \theta - x_m x_m s^l \\
		 &=2i\lambda _p x_l (\tfrac{1}{(2i\lambda_p)^2}x_i \extd x_i) - x_m x_m s^l
		 =\tfrac{1}{2i\lambda_p} x_l x_i \extd x_i - x_m x_m s^l \end{align*} using the two forms of $\theta$. We then use the sphere relation and rearrange as stated. \endproof

\medskip
\section{Moduli of QLCs on the fuzzy sphere}\label{secqlc}

Once we have fixed the calculus it is clear, since $\Omega^1(\C_\lambda[S^2])$ has a central basis, a general metric, as already observed in \cite[Example~1.46]{BegMa}, takes the form
\[ g=g_{ij}s^i\tens s^j\]
where, since $g$ has to be central, we need the coefficients $g_{ij}$ to be central. Since the centre of $U(su_2)$ is generated by the quadratic Casimir, it follows that $\C_\lambda[S^2]$ has trivial centre, so $g_{ij}\in \C$. For quantum symmetry we clearly need $g_{ij}=g_{ji}$ and for the reality property we note that $s^i{}^*=s^i$ so that we need $g_{ij}$ to be hermitian which, given the symmetry, means $g_{ij}\in \R$. Finally, we need $g_{ij}$ to be an invertible matrix with inverse $g^{ij}$, say. Then $(s^i,s^j)=g^{ij}$. The new question, which we address, is what are the QLCs and WQLCs. Part of this was answered in \cite{Msc}, although not our main result Proposition~\ref{QLC}.

\subsection{Quantum Levi-Civita connection}

First, we consider an arbitrary connection in the equivalent forms
\[ \nabla s^i= -\tfrac{1}{2}\tchrs{i}{jk}s^j \otimes s^k , \]
where $\tchrs{i}{jk}\in A$.  As $s^i$ form a basis, any such form gives a left connection. 

\begin{proposition}\label{torcotor} Let $\Gamma_{ijk}:=g_{im}\Gamma^m{}_{jk}$. 

\begin{enumerate}
\item $\nabla$ is torsion free if and only if  $\Gamma_{ijk}-\Gamma_{ikj}= 2g_{im}\eps_{mjk}$.

\item $\nabla$ is cotorsion free if and only if $\tchrs{}{ijk} -\tchrs{}{jik}=2g_{km}\epsilon_{mij}$.
\end{enumerate}
\end{proposition}
\proof (1) is immediate from 
	\begin{align*}
		T_\nabla s^i= \wedge \nabla s^i - \extd s^i& = \wedge (-\tfrac{1}{2}\tchrs{i}{kl}s^k \otimes s^l) + \tfrac{1}{2}\epsilon_{ikl}s^k \wedge s^l 
	\end{align*}
	that torsion freeness needs $
        \tchrs{i}{kl} - \tchrs{i}{lk} +2\epsilon_{ilk}=0$, which we write to match (2) in terms of the lowered index version.
        
(2) We calculate
		\begin{align*}
		coT_{\nabla} & =\extd (g_{ij}s^i)\otimes s^j -g_{ij}s^i \wedge (-\tfrac{1}{2}\tchrs{j}{kl}s^k\otimes s^l ) \\
		& = g_{ij}(-\tfrac{1}{2}\epsilon_{imn}s^m \wedge s^n ) \otimes s^j + \tfrac{1}{2} g_{ij}s^i \wedge (\tchrs{j}{kl}s^k \otimes s^l ) \\
		& = -\tfrac{1}{2}( g_{ml}\epsilon_{mik} - g_{ij}\tchrs{j}{kl})(s^i \wedge s^k )\otimes s^l.
		\end{align*}	
 So vanishing of cotorsion requires that  $g_{ml}\epsilon_{mik} - g_{ij}\tchrs{j}{kl}$ is symmetric in $i,k$, which we have written as stated. 
\endproof
Therefore, the moduli of WQLCs (where torsion and cotorsion vanish) are given by the two conditions (1) and (2) simultaneously. Note that the two conditions have a very similar form, which is part of the symmetric role of torsion and cotorsion in the WQLC theory. Finally, for the full  QLC theory, we need to know when $\nabla$ is a bimodule connection, i.e. when there exists a suitable generalised braiding $\sigma$ and what it looks like. 

\begin{lemma} \label{sigma} If $\nabla$ is a bimodule connection on $\Omega^1(\C_\lambda[S^2])$ then 
\[ \sigma(s^i \otimes s^j )= s^j \otimes s^i+\tfrac{1}{2(1-\lambda _p^2)} ( \tfrac{1}{2i\lambda _p} x_j x_n [\tchrs{i}{lk},x_n] + \epsilon_{jmn} [\tchrs{i}{lk},x_m]x_n )s^l \otimes s^k  \]
\end{lemma}
\proof The generalised braiding if it exists, given that $[s^i , x_j ]=0$, is characterised by
	\[ \sigma (s^i \otimes \extd x_j )=\extd x_j \otimes s^i -[\nabla s^i , x_j ]= \epsilon_{jmn}x_m s^n \otimes s^i + \tfrac{1}{2} [\tchrs{i}{mn},x_j]s^m \otimes s^n \]
where the tensor is over the algebra $A$. 
Now substituting  $s^n$ from Lemma~\ref{sl},  
\begin{align*}
		\sigma(s^i \otimes s^n )&=\sigma( s^i \otimes \tfrac{1}{2i\lambda_p (1-\lambda_p^2 )} x_n x_j\extd x_j + \tfrac{1}{1-\lambda_p^2}\epsilon _{nmj}\extd x_m x_j) \\ 
		& =\tfrac{1}{2i\lambda_p (1-\lambda_p^2 )} x_n x_j \sigma(s^i \otimes \extd x_j )  + \tfrac{1}{1-\lambda_p^2}\epsilon _{nmj} \sigma(s^i \otimes \extd x_m ) x_j\\
		&= \tfrac{1}{2i\lambda_p (1-\lambda_p^2 )} x_n x_j (\epsilon_{jmk}x_m s^k \otimes s^i + \tfrac{1}{2}[\tchrs{i}{mk},x_j]s^m \otimes s^k) \\
		&\qquad+ \tfrac{1}{1-\lambda_p^2}\epsilon_{nmj} (\epsilon_{mlk}x_l s^k \otimes s^i + \tfrac{1}{2}[\tchrs{i}{lk},x_m]s^l \otimes s^k)x_j\\
		& =\tfrac{1}{1-\lambda _p^2}(\tfrac{1}{2i\lambda _p} \epsilon_{jmk}x_n x_j x_m + \epsilon_{nmj} \epsilon_{mlk} x_l x_j )s^k \otimes s^i +\\
		&\qquad  + \tfrac{1}{2(1-\lambda _p^2)} ( \tfrac{1}{2i\lambda _p} x_n x_j [\tchrs{i}{lk},x_j] + \epsilon_{nmj} [\tchrs{i}{lk},x_m]x_j )s^l \otimes s^k \\
		& = s^n\tens s^i+ \tfrac{1}{2(1-\lambda _p^2)} ( \tfrac{1}{2i\lambda _p} x_n x_j [\tchrs{i}{lk},x_j] + \epsilon_{nmj} [\tchrs{i}{lk},x_m]x_j )s^l \otimes s^k 
	\end{align*}
as stated, using the commutation relations in $A$.  \endproof

If follows that $\sigma$ is a flip if the $\Gamma$ are central, i.e. the case of constant coefficients $\Gamma_{ijk}\in\C$. At least in this case, $\sigma$ will be manifestly well defined as a bimodule map, otherwise this will depend on the commutators with $\Gamma$. Also in this flip case the `reality' property with respect to $*$ then reduces to $\Gamma_{ijk}\in\R$. We are now ready to consider the condition for full metric compatibility.

\begin{lemma}\cite{Msc} The metric compatibility condition assuming $\Gamma$ are constant coefficients is
	\[ \tchrs{}{lik} + \tchrs{}{kil} = 0\]	
	and in this case $\sigma(s^i\tens s^j)=s^j\tens s^i$ is well-defined. 
	\end{lemma}
  \begin{proof}	
	\begin{align*}
		\nabla g &=(\nabla \otimes \id )(g_{ij}s^i \otimes s^j) + (\sigma \otimes \id )(g_{ij}s^i \otimes \nabla(s^j )) \\
		& = -g_{ij}\tfrac{1}{2}\tchrs{i}{mn}s^m \otimes s^n \otimes s^j - (\sigma \otimes \id )(g_{ij}s^i \otimes \tfrac{1}{2}\tchrs{j}{mn}s^m \otimes s^n )  \\
		& = -g_{ij}\tfrac{1}{2}\tchrs{i}{mn}s^m \otimes s^n \otimes s^j - \tfrac{1}{2}g_{ij}\tchrs{j}{mn}\sigma(s^i \otimes s^m )\otimes s^n \\
		& = -\tfrac{1}{2}(\tchrs{}{nmi}+\tchrs{}{imn})s^m \otimes s^i \otimes s^n - \tfrac{1}{4(1-\lambda_p^2)} \tchrs{}{imn}\left(\tfrac{x_mx_p}{2\imath\lambda_p} [\tchrs{i}{lk}, x_p ]+ \epsilon_{mpq}[\tchrs{i}{lk}, x_p]x_q\right) s^l \otimes s^k \otimes s^n. 	
	\end{align*}	
In the natural case of constant coefficients of $\Gamma$, we can drop the second term. \end{proof}

It is a nice check that torsion free and metric compatible (in our constant $\Gamma$ case) implies cotorsion free, as it must. Indeed, we can write torsion freeness as $\tchrs{}{lki}- \tchrs{}{lik}  = 2g_{lm}\epsilon_{mki} $ by Proposition~\ref{torcotor}. Given metric compatibility in the form just found, this is equivalent to $ - \tchrs{}{ikl} +\tchrs{}{kil} = 2g_{lm}\epsilon_{mki}$, which is the cotorsion free condition in Proposition~\ref{torcotor}. 
It remains to solve for the moduli of constant coefficient QLC solutions for a given  a metric $g_{ij}s^i \otimes s^j$. 

\begin{proposition}\label{QLC} For any metric $g_{ij}$, there is a unique  QLC among those with constant coefficients, namely 
\[ \Gamma_{ijk}=2\eps_{ikm}g_{mj}+{\rm Tr}(g)\eps_{ijk}.
 \]
There are real, hence the connection is $*$-preserving. 
\end{proposition} 
\proof We have to  solve the joint system
\[ \tchrs{}{ikl} - \tchrs{}{ilk} =2g_{im}\epsilon_{mkl} , 
	\quad \tchrs{}{lik} + \tchrs{}{kil} = 0,\]
	the second of which is solved by setting $\Gamma_{ijk}=\eps_{ikm}\gamma_{mj}$ for some matrix $\gamma$. Letting $L_i$ be the matrices $(L_i)_{mn}=\eps_{imn}$, the first equation is then
	\[ L_i\gamma-(L_i\gamma)^t+2g_{im}L_m=0\]
	as matrices. This is a  linear system for $\gamma$ with a unique solution
	\[ \gamma=2g-{\rm Tr}(g)\id\]
	which translates into the solution stated. 
	% \left(
%\begin{array}{ccc}g_{11}-g_{22}-g_{33} & 2 g_{12} & 2 g_{13} \\
 %2 g_{12} & - g_{11}+ g_{22}-g_{33} & 2 g_{23} \\
 %2 g_{13} & 2 g_{23} & - g_{11}- g_{22}+g_{33} \\
%\end{array}\right)  \]
	 (Note that we do
	not have to solve the cotorsion equation $\tchrs{}{kil} -\tchrs{}{ikl} +2g_{lm}\epsilon_{mki}= 0$ as this is implied, as mentioned.) \endproof
For example, when $g_{ij}=\delta_{ij}$ (the rotationally invariant or `round') metric, we have a unique solution   $\Gamma_{ijk}=\eps_{ijk}$.

\subsection{Ricci curvature}\label{seccurv}  

\medskip
Now that we understand the moduli of QLCs, we explore their curvature on the fuzzy sphere for general metrics.  The curvature in quantum Riemannian geometry is defined by 
\[ R_\nabla:\Omega^1\to \Omega^2\tens_A\Omega^1,\quad R_\nabla  = (\extd \otimes \id-\id \wedge \nabla)\nabla\]  
which in our case we can necessarily write in the form
\[ R_\nabla(s^i)=\rho^i{}_{jk}\eps_{jmn}s^m\wedge s^n\tens s^k\]
for some curvature coefficients $\rho^i{}_{jk}\in A$.  We are also interested in taking a `trace' for the  Ricci tensor and the Ricci scalar, which in the current framework \cite{BegMa} means with respect to a further, but in our case canonical, `lift' map 
\[ i:\Omega^2\to \Omega^1\tens_A\Omega^1,\quad  i(s^i \wedge s^k) = \tfrac{1}{2}(s^i \otimes s^k - s^k \otimes s^i ). \]
For this choice of $i$ and form of $R_\nabla$, we have 
\begin{align*}{\rm Ricci}&=((\ ,\ )\tens\id)(\id\tens i\tens\id)(\id\tens R_\nabla)(g)\\
&=\tfrac{1}{2}g_{pi}\rho^i{}_{jk}\eps_{jmn}((s_p,\ )\tens\id)(s^m\tens s^n\tens s^k-s^n\tens s^m\tens s^k)\\
&=\tfrac{1}{2}\rho^i{}_{jk}\eps_{jin}s^n\tens s^k-\tfrac{1}{2}\rho^i{}_{jk}\eps_{jmi}s^m\tens s^k	= \rho^i{}_{jn}\eps_{jim}s^m\tens s^n. 
\end{align*}
Hence the Ricci tensor defined by ${\rm Ricci}=R_{mn}s^m\tens s^n$ and Ricci scalar $S=(\ ,\ ){\rm Ricci}$ are 
\[  R_{mn}=\rho^i{}_{jn}\eps_{jim},\quad S=\rho^i{}_{jn}\eps_{jim}g^{mn}.\]
Finally,  {\em we adopt the convention that indices of $\eps$ can be raised with the inverse metric $(\ ,\ )=g^{-1}$ with matrix entries  $g^{ij}$}.

\begin{proposition}\label{propS} For $\C_\lambda[S^2]$, the scalar curvature for the QLC in Proposition~\ref{QLC} is
\[ S= \tfrac{1}{2}({\rm Tr}(g^2)-\tfrac{1}{2}{\rm Tr}(g)^2)/\det(g). \]
\end{proposition}
\proof We first compute from its definition as given above that
 \begin{align*} 
		R_{\nabla}(s^i) &  =-\tfrac{1}{2}(\tchrs{i}{kl}\extd s^k +(\extd \tchrs{i}{kl})s^k )\otimes s^l + \tfrac{1}{2}\tchrs{i}{kl}s^k\wedge\nabla s^l\\
		  & =\tfrac{1}{4}\tchrs{i}{kl}\epsilon_{kmn}s^m \wedge s^n\tens s^l - \tfrac{1}{2}\extd \tchrs{i}{kl}\wedge s^k  \otimes s^l  -\tfrac{1}{4}\tchrs{i}{kl}\tchrs{l}{mn} s^k\wedge s^m \otimes s^n
	\end{align*}	
which corresponds to
\begin{equation}\label{rho}  \rho^i{}_{jk}=\tfrac{1}{4}\tchrs{i}{jk}-\tfrac{1}{4}\eps_{jmn}\del_m\tchrs{i}{nk}-\tfrac{1}{8}\eps_{jmn}\tchrs{i}{ml}\tchrs{l}{nk}\end{equation}
where $\extd f=(\del_i f)s^i$ defines the partial derivatives. One can check that this reproduces the same $R_\nabla$.
This applies for any left connection on $\Omega^1$.  

We now specialise to the QLC in Proposition~\ref{QLC}. Then
\begin{align*} \rho^i{}_{jk}&=\tfrac{1}{2}(\eps^i{}_{km}g_{mj}-\eps^i{}_{lp}\eps_{jmn}\eps^l{}_{kq}g_{pm}g_{qn})+\tfrac{1}{4}{\rm Tr}(g)( \eps^i{}_{jk}- \eps^j{}_{ik}+\eps^{ij}{}_pg_{pk})-\tfrac{1}{8}{\rm Tr}(g)^2\eps^{ij}{}_k
\end{align*}
which we contract  to obtain Ricci as
\[ R_{st}=\tfrac{1}{2}\eps^i{}_{tm}g_{mj}\eps_{jis}+ \tfrac{1}{2}{\rm Tr}(g^{-1})g_{st}-\tfrac{1}{2}\delta_{st}+\tfrac{1}{2}{\rm Tr}(g)(g^{st}-{\rm Tr}(g^{-1})\delta_{st}-\tfrac{1}{2}\eps^{ij}{}_p\eps_{ijs}g_{pt})+\tfrac{1}{8}{\rm Tr}(g)^2\eps^{ij}{}_t\eps_{ijs}.\]
This then contracts further to 
\[ S={\rm Tr}(g^{-1})+\tfrac{1}{2}\eps^{ij}{}_m (g_{mk}-\tfrac{1}{2}{\rm Tr}(g)\delta_{mk}+\tfrac{1}{4}{\rm Tr}(g)^2 g^{mk})\eps_{ijk}+\tfrac{1}{2}{\rm Tr}(g)({\rm Tr}(g^{-2})-{\rm Tr}(g^{-1})^2). \]
Finally, we identify the middle $\eps...\eps$ expression in terms of $\det g^{-1}$. This is most easily seen assuming that $g-=\rm{diag}(\lambda_1,\lambda_2,\lambda_3)$, say, but then holds generally. Here $\eps^{ij}{}_{k}g^{mk}\eps_{ijk}=\eps_{ijk}^2\lambda_i^{-1}\lambda_j^{-1}\lambda_k^{-1}=6\det g^{-1}$ (summing over $i,j,k$) is well-known, but a similar method gives 
\[ \eps^{ij}{}_{k}\eps_{ijk}=\eps_{ijk}^2\lambda_i^{-1}\lambda_j^{-1}\eps_{ijk}^2\lambda_i^{-1}\lambda_j^{-1}\lambda_k^{-1}\lambda_k=2\det g^{-1}{\rm Tr}(g)\]
\[ \eps^{ij}{}_{m}g_{mk}\eps_{ijk}=\eps_{ijk}^2\lambda_i^{-1}\lambda_j^{-1}\lambda_k=\eps_{ijk}^2\lambda_i^{-1}\lambda_j^{-1}\lambda_k^{-1}\lambda_k^2=2\det g^{-1}{\rm Tr}(g^2).\] In this way, we obtain
\[ S={\rm Tr}(g^{-1})+\left(  {\rm Tr}(g^2)+\tfrac{1}{4}{\rm Tr}(g)^2\right)\det g^{-1} + \tfrac{1}{2}{\rm Tr}(g)\left( {\rm Tr}(g^{-2}) - {\rm Tr}(g^{-1})^2\right)\]
which then simplifies further to the form stated, as one can again check in the diagonal case. In terms of the $\lambda_i$, this is 
\begin{equation}\label{Slambda} S={\lambda_1^2+\lambda_2^2+\lambda_3^2- 2( \lambda_1\lambda_2+\lambda_1\lambda_3+\lambda_2\lambda_3)\over 4\lambda_1\lambda_2\lambda_3},\end{equation}
which one can also regard as a function on the space of metrics modulo conjugation if we think of the $\lambda_i$ as the (not necessarily distinct) eigenvalues of $g$. \endproof

 Note  that the classical limit of $S$ in the noncommutative geometry conventions here is $-{1\over 2}$ of the classical value. Hence a classical unit sphere with its usual round metric would in our conventions have  $S=-1$. In the fuzzy case, for the rotationally invariant `round' metric $g_{ij}=\delta_{ij}$ on $\C_\lambda[S^2]$ and the unique QLC $\Gamma_{ijk}=\eps_{ijk}$ in Proposition~\ref{QLC}, we have 
\[ R_{mn}=-\tfrac{1}{4}\delta_{mn},\quad S=-\tfrac{3}{4}.\]
If we perturb around this metric by setting $g=\id+\eps$ then we can write
\begin{align*} S&= -\tfrac{3}{4} + \tfrac{1}{4}{\rm Tr}\eps - \tfrac{1}{12}({\rm Tr}\eps)^2 +\tfrac{1}{4}(\eps_{12}^2+\eps_{13}^2+\eps_{23}^2)\\
&\quad+ \tfrac{1}{24}\left((\eps_{11}- \eps_{22})^2+(\eps_{11}-\eps_{33})^2+(\eps_{22}-\eps_{33})^2\right)+ O(\eps^3),\end{align*}
showing an unbounded  mode for the average of the diagonal entries  plus a positive definite part for the `fluctuations' off diagonal or between the diagonal entries. 

\section{Euclideanised quantum gravity on the fuzzy sphere}\label{secqg}

The scalar curvature found above is the main, but not the only, ingredient for quantum gravity in a functional integral approach. Here we briefly consider the other elements and formulate the theory, although it will remain too hard to compute explicitly. Moreover, this will not be physical gravity since our spherical form of coordinates where $\sum_i x_i^2=1-\lambda_p^2$ are suitable for a Euclidean signature (the actual `shape' depends on the metric but the choice of algebra plays the role of the manifold in some sense) but this is still of interest in certain contexts. In the Euclidean case, the partition function for the functional integral should take the form  
\[ Z=\int {\mathcal D}g\, e^{-{2\over G} \int S[g]} \]
(and similarly with operators inserted for expectation values), where $G$ is some real positive coupling constant and
\begin{equation}\label{Sg} S[g]=\frac{g_{11}^2+g_{22}^2+g_{33}^2-2\left( g_{11} g_{22}+g_{11}g_{33}+g_{22}g_{33}\right)+4 \left(g_{12}^2+g_{13}^2+g_{23}^2\right)}{4\det g}\end{equation}
 is the scalar curvature in Proposition~\ref{propS}. 
  
 We also need for the action a map $\int: A\to \C$ which classically would be the Lebesgue measure in local spacetime coordinates times $\sqrt{|\det(g)|}$. One could attempt to characterise this map in the quantum case by requiring that it is a positive linear functional (so $\int a^* a\ge 0$ for all $a$ in the algebra) compatible with the Riemannian metric in some way (classically this would be so as to vanish on a total divergence). This is an interesting question but is not needed for our present purposes as $S[g]\in \R$ is a multiple of the constant function, so the only thing we need is $\int 1$. The natural choice here similar to integration on a Riemannian manifold would be power of $|\det g|$ (but not necessarily its square root as classically, given that our tangent space and metric have the wrong dimension compared to classical sphere). Bearing in mind the $\det(g)$ in the denominator of $S[g]$, the most natural choice here is to cancel this by setting
 \[ \int 1=|\det(g)|.\]
 
 Finally, we need a measure ${\mathcal D}g$ for integration over the space of metrics. In our case,  this is the 6-dimensional space $\CP_3$ of $3\times 3$ positive-definite symmetric matrices, which has the natural structure of a Riemannian manifold as the noncompact symmetric space $GL_3(\R)/O_3(\R)$, with an invariant metric $\cg_{\CP_3}$ given in line element terms as 
 \[ \extd s^2={\rm Tr}( (g^{-1}\extd g)^2).\]
Integration over $\CP_3$ therefore has a canonical Riemannian manifold measure ${\mathcal D}g$ defined relative to Lebesgue measure in local  coordinates by $\sqrt{|\det( \cg_{\CP_3})|}$. The latter works out at metric $g$ as a factor $|\det(g)|^{-2}$, see \cite[Sec 4.1.3]{Ter}. 

Putting these ingredients together, we thus define Euclideanised quantum gravity on the fuzzy sphere by partition function 
 \begin{equation}\label{ZP3} Z= \int_{\CP_3} \prod_{i\le j}\extd g_{ij}\  |\det(g)|^{-2}\,   e^{-{1\over  G}({\rm Tr}(g^2)-\frac{1}{2}{\rm Tr}(g)^2)},\end{equation} 
 where $G$ is a real coupling constant. Here expectation values are the ratio of the same expression with operators inserted divided by $Z$, both parts of which can be expected to diverge given the above initial remarks and the noncompact nature of $\CP_3$. One can also write 
 \[ \det(g)^{-2}=(2\pi)^{-{3\over 2}}\int_{\R^3}\extd^3 \vec x\ e^{-{1\over 2}\vec x^t g^2 \vec x}\] so that up to a discarded constant and assuming we can swap the order of integration,
 \begin{equation}\label{Zx} Z=\int_{\R^3}\prod_{i}\extd x_i \int_{\CP_3} \prod_{i\le j}\extd g_{ij}\   e^{-{1\over  G}({\rm Tr}(g^2)-\frac{1}{2}{\rm Tr}(g)^2)-{1\over 2}\vec x^t g^2 \vec x}, \end{equation}
giving an idea of the formal content of the theory. Moreover, {\em if} we ignored the restriction to $\CP_3$, we could then do the $\extd g$ integration as a Gaussian to give the inverse determinant of a quadratic form built from the $x_i$. We can also write $(\det g)^{-2}=e^{-2{\rm Tr}\ln g}$ in (\ref{ZP3}) as a non-quadratic `interaction' term. We now look more closely at the theory in a reduced form where we look only at $SO_3$-invariant expressions. 
  
Our first step it to parameterize positive symmetric matrices according to the spectral decomposition  $g=C^t{\rm diag}(\lambda_1,\lambda_2,\lambda_3)C$ for some $\vec\lambda\in \R_{>0}^3$ and some $C\in SO_3$. The latter is not unique but the multiplicity is discrete and generically we can convert $\extd g$ to these new coordinates with an appropriate Jacobean. To do this explicitly, we let $E(\theta,\phi,\psi)$ be the Euler rotation matrix for angles $\theta,\phi,\psi$ and let
\[ g=E(\theta,\phi,\psi)^t {\rm diag}(\lambda_1,\lambda_2,\lambda_3)E(\theta,\phi,\psi).\]
This change of variables is locally invertible for $\sin \phi\ne 0$ and distinct $\lambda_i$, indeed the Jacobean can be computed and we find
\[  \prod_{i\le j}\extd g_{ij} =\extd \theta\, \extd\phi\, \extd\psi\, |\sin(\phi)| \prod_i\extd\lambda_i\  | (\lambda_1-\lambda_2)(\lambda_1-\lambda_3)(\lambda_2-\lambda_3)|\]
On the other hand, the action itself does not depend on the $SO_3$ conjugation (having the value in (\ref{Slambda}) but without the $\lambda_1\lambda_2\lambda_3$ denominator there), so for the partition function and for any insertions that depend only on the $\lambda_i$ and not on the angles, we can do the integration over a dense subset of $SO_3$ to give a constant, which we ignore. Hence (\ref{ZP3}) becomes effectively
\begin{equation}\label{Zlam} Z=\int_\eps^L\kern-5pt\prod_i\extd\lambda_i\,  \frac{| (\lambda_1-\lambda_2)(\lambda_1-\lambda_3)(\lambda_2-\lambda_3)|}{\lambda_1^2\lambda_2^2\lambda_3^2}e^{-\frac{1}{2G}(\lambda_1^2+\lambda_2^2+\lambda_3^2- 2( \lambda_1\lambda_2+\lambda_1\lambda_3+\lambda_2\lambda_3))},\end{equation}
where we have introduced cut-offs $L>>\eps>0$ to regulate divergences at both ends. The story here turns out to be very similar to \cite{Ma:squ} in that the divergence as $\eps\to 0$ does not show up when we look as vacuum expectation values as these are ratios (both top and bottom diverge as $\eps\to 0$ but the ratio is well defined in this limit). Moreover,
\[  \<\lambda_{i_1}\cdots\lambda_{i_n}\>\sim\frac{3}{16}L^n\]
for large $L$,  independently of which $\lambda_i$ are involved (this was checked numerically to several orders of $\lambda$), as plotted for $n=1,2$ in Figure~\ref{fuzqg}.  It follows that we have well-defined ratios
\[ {\<\lambda_{i_1}\cdots\lambda_{i_n}\>\over\<\lambda_i\>^n}=(\tfrac{16}{3})^{n-1}\]
in the limit $L\to\infty$, and in particular that there is a uniform relative uncertainty in the $\lambda_i$,
 \[ {\Delta\lambda_i\over \<\lambda_i\>}:={\sqrt{\<\lambda_i^2\>-\<\lambda_i\>^2}\over  \<\lambda_i\>}=\sqrt{\tfrac{13}{3}}\]
 similarly to quantum gravity on a square in \cite{Ma:squ}. Note, however, that these are formal interpretations given that this is a Euclidean theory. 
\begin{figure}
\[ \includegraphics[scale=0.47]{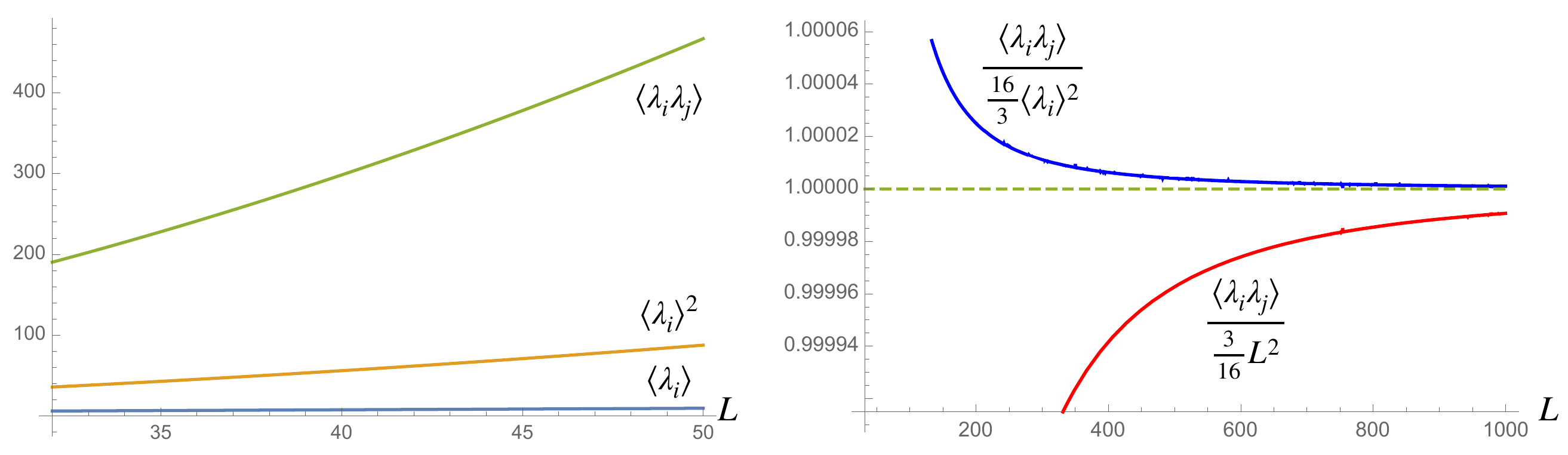}\]
\caption{Expectation values for Euclidean quantum gravity on the fuzzy sphere as a function of the cutoff $L$. The right graph shows $\<\lambda_i\lambda_j\>$ converging to ${16\over 3}\<\lambda_i\>^2$ and ${3\over 16}L^2$. \label{fuzqg}}
\end{figure}

We can also follow the pattern of \cite{Ma:squ} and look at a partial theory where we regard the average of the $\lambda_i$ as a background metric with respect to which we are quantising only the differences. Thus, we let 
\[ \lambda_1=u-2v,\quad \lambda_2=u+v-w,\quad \lambda_3=u+v+w\]
with inverse
\[ u=\tfrac{1}{3}(\lambda_1+\lambda_2+\lambda_3),\quad v=\tfrac{1}{6}(\lambda_2+\lambda_3-2\lambda_1),\quad w=\tfrac{1}{2}(\lambda_3-\lambda_2),\]
which diagonizes the quadratic form in the action to 
\[ \lambda_1^2+\lambda_2^2+\lambda_3^2- 2( \lambda_1\lambda_2+\lambda_1\lambda_3+\lambda_2\lambda_3)=-3u^2+12 v^2+4w^2.\]
In the partial theory, we leave out the $\extd u$ integral and regard $u>0$ as a parameter. Putting in the region of $u,v,w$ corresponding to $\lambda_i>0$, we have the effective theory for the `fluctuation' variables $v,w$,
\begin{align*} Z_u&=2\int_{-u}^{u\over 2}\extd v\int_{-(u+v)}^{u+v} \extd w\, {|(9v^2-w^2)w|\over (u-2v)^2((u+v)^2-w^2)^2}e^{-{2\over G}(3 v^2+w^2)}\\
&=4\int_{-u}^{u\over 2}\extd v\int_{0}^{u+v} \extd w\, {|9v^2-w^2|w\over (u-2v)^2((u+v)^2-w^2)^2}e^{-{2\over G}(3 v^2+w^2)}.
\end{align*}
This is still divergent at the boundaries corresponding previously to $\lambda_i=0$, but has the merit that the inner integral can now be done analytically. The previous partition function 
\[ Z=\int_0^\infty \extd u\  e^{{3\over 2G}u^2} Z_u\]
still contains the  other divergence at $\lambda_i=\infty$, requiring a cut-off.  One can make this change for the diagonal entries of any metric and have the off-diagonals as three further Gaussian variables according to (\ref{Sg}), but the restriction on the variables for a positive metric is then much harder
to describe.

\section{Fuzzy monpole}  \label{secmon}

We have focussed on the quantum Riemannian geometry of the fuzzy sphere with its new calculus from \cite{BegMa}. However, we could also
ask about the monopole connection for this calculus. Classically, this arises naturally on the tautological line bundle over the sphere which in
algebraic terms is the Grassmann connection for the rank 1 projector associated with that. It was already shown in \cite{Ma:fuz} that
the fuzzy sphere (as well as the $q$-fuzzy sphere) has such a projector giving a natural rank 1 line bundle $\CS$ and it is already explained in \cite[Example~3.27]{BegMa} how one may then compute the Grassmann connection, depending on the choice of calculus. Here we do this explicitly for our choice of $\Omega(\C_\lambda[S^2])$. 

First, it can be convenient to use new `coordinates' as in \cite{Ma:fuz,BegMa} with relations
\[ x=\tfrac{1}{2} (x_3+1+\lambda_p),\quad z=\tfrac{1}{2}(x_1+\imath x_2);\quad [x,z]=\lambda_p z,\quad [z,z^*]=2\lambda x-\lambda_p(1+\lambda_p),\quad z^*z=x(1-x)\]
and let \cite{Ma:fuz}
\[ P=\begin{pmatrix} 1+\lambda_p-x & z\\ z^* & x\end{pmatrix}=\tfrac{1}{2}\begin{pmatrix} 1+\lambda_p-x_3 & x_1+\imath x_2\\ x_1-\imath x_2 & 1+\lambda_p+x_3\end{pmatrix}\]
which one can check obeys $P^2=P$. We then define a projective basis of $\CS$ with relation,
\[ e^1=(1,0).P=(1+\lambda_p-x, z),\quad e^2=(0,1).P=(z^*,x);\quad (x-\lambda_p)e^1=z e^2\]
as one can check (so these are not independent, the bundle is really 1-dimensional). Here, we will stick with 
the $x_i$ generators in order to match the rest of the paper. 

\begin{lemma} The Grassman connection $\nabla_\CS e^\alpha=(\extd P_{\alpha\beta})P_{\beta\gamma}\tens e^\gamma$ on $\CS$ for our calculus on $\C_\lambda[S^2]$ is 
\begin{align*}
    \nabla_\CS e^\alpha & = \left({1+ \lambda_p \over 2} \extd P_{\alpha \beta} +\lambda_p  P_{\alpha \beta}\theta +\imath\frac{1-\lambda_p^2}{4} Q_{\alpha \beta} -\frac{\lambda_p(1-\lambda_p)}{2}\delta_{\alpha\beta}\theta\right)\otimes e^\beta,
\end{align*}
where  $Q =\begin{pmatrix}
   - s^3  & s^1+\imath s^2 \\
s^1-\imath s^2 &  s^3
\end{pmatrix}$. 
\end{lemma}
\proof Here $\nabla_\CS:\CS\to \Omega^1\tens_A\CS$ obeys similar axioms to those for a left connection on $\Omega^1$ in Section~\ref{secpre}. The calculation of the stated formula from the projector is a straightforward from
\[ \extd P = \tfrac{1}{2}\begin{pmatrix}
x_2 s^1 - x_1 s^2 & x_2 s^3 - x_3 s^2 + \imath x_3 s^1 -\imath x_1 s^3  \\
x_2 s^3 - x_3 s^2 -\imath x_3 s^1 +\imath x_1 s^3 & x_1 s^2 - x_2 s^1 
\end{pmatrix}\]
and the commutation relations in the algebra. For example,
\begin{align*} 
    (\extd P.P)_{11} & = \tfrac{1}{4}\left( (x_2 s^1 - x_1 s^2)(1+\lambda_p - x_3) + (x_2 s^3 - x_3 s^2 -\imath x_3 s^1 +\imath x_1 s^3)(x_1 - \imath x_2 ) \right)\\
    & = \tfrac{1}{4}\left( (1+\lambda _p )(x_2 s^1 - x_1 s^2) + \imath x_3 (x_1 s^1 + x_2 s^2 + x_3 s^3) - \imath (1-\lambda_p^2 ) s^3 + \right. \\
    & \phantom{= \tfrac{1}{4} ( } \left. + [x_3, x_2]s^1 + [x_1 , x_3 ]s^2 + [x_2 , x_1 ]s^3 \right) \\
    & = \tfrac{1}{4}\left( (1+\lambda _p )(x_2 s^1 - x_1 s^2 )-2x_3 \lambda_p \theta +\imath (\lambda_p^2 -1)s^3 - 2\imath \lambda_p (x_1 s^1 + x_2 s^2 + x_3 s^3 ) \right)\\
    & = \tfrac{1}{4}\left( (1+ \lambda _p )\extd (- x_3) +\imath (\lambda_p^2 -1)s^3 - 2\lambda_px_3 \theta + 4\lambda_p^2\theta\right) \\
    & = \tfrac{1}{2} (1+ \lambda _p )\extd P_{11} + \lambda_p P_{11} \theta -\tfrac{1}{2} (1-\lambda_p)\lambda_p\theta -\tfrac{\imath}{4} (1-\lambda_p^2)s^3.
\end{align*}
Similary for the other entries of $\extd P.P$, the proof of which we omit. 
\endproof

%For the curvature, we use the above for $\extd P.P$ to compute
%\begin{align*}
 %   \extd P&\wedge(\extd P)P  ={1+\lambda_p \over 2}\extd P\wedge\extd P+\lambda_p \extd P\wedge P\theta+\imath\frac{1-\lambda_p^2}{4}\extd P\wedge Q -\frac{\lambda_p(1-\lambda_p)}{2}  \extd P.\theta\\
 %   &={1+\lambda_p \over 2}\extd P\wedge\extd P+\lambda_p^2(\extd P.\theta+P\theta^2)+\imath\frac{\lambda_p(1-\lambda_p^2)}{4}Q\theta -\frac{\lambda_p^2(1-\lambda_p)}{2}\theta^2+\imath\frac{1-\lambda_p^2}{4}\extd P\wedge Q, 
%\end{align*}
%where the second term can also be rewritten using $\extd P\theta+P\theta^2=\theta P
%\theta=\theta^2P-\theta\extd P$. 

The curvature similarly acts as a  2-form valued matrix on our basis vectors $e^\alpha$, this time given by 
\[\extd P\wedge(\extd P)P= \frac{ \imath  (1-\lambda_p )}{4}\left(f_{12}s^1\wedge s^2+ f_{31}s^3\wedge s^1+ f_{23}s^2\wedge s^3\right)\]
for some $A$-valued matrix coefficients. For example, one can compute
\[ f_{12}= \begin{pmatrix}
 (x_3-\lambda_p)(1+\lambda_p-x_3) &  (x_3-\lambda_p)(x_1+\imath x_2)   \\
  (x_3+\lambda_p) (x_1-\imath  x_2) &(x_3+\lambda_p)(1+\lambda_p+x_3)  
\end{pmatrix}=2\begin{pmatrix}
 x_3-\lambda_p & 0  \\
 0&x_3+\lambda_p
\end{pmatrix}P  \]
\[ f_{31}=
 \begin{pmatrix} x_2(1+\lambda_p-x_3) +\imath\lambda_p(x_1-\imath x_2)
 &   x_2 (x_1+\imath x_2)+\imath \lambda_p  (1+\lambda_p +x_3) \\
-\imath \lambda_p  (1+\lambda_p -x_3)+ x_2( x_1-\imath x_2) & -\imath\lambda_p(x_1+\imath x_2)+x_2(1+\lambda_p+x_3)\end{pmatrix}=2\begin{pmatrix}
 x_2&\imath\lambda_p \\ -\imath\lambda_p  &x_2
\end{pmatrix}P.\]  

% \[ f_{23}= \begin{pmatrix}x_1(2\lambda_p-x_3)+x_1+\imath x_2& x_1(x_1+\imath x_2)+\lambda_p  (1+\lambda_p +x_3) \\
%x_1(x_1-\imath x_2)+\lambda_p  (1+\lambda_p-x_3) &   x_1-\imath x_2+x_1(2\lambda_p+x_3)
%\end{pmatrix} \]

\section{Concluding remarks}\label{secrem}

Fuzzy-$\R^3$ in the form of the angular momentum algebra $U(su_2)$ has a long history as a `quantisation' as well as clear applications such as at the heart of 3D Euclideanised quantum gravity without cosmological constant. Its quotient the fuzzy sphere is likewise well known as the quantisation of a coadjoint orbit. Although less relevant perhaps to 3D quantum gravity, it is nevertheless related to Penrose' spin network geometry\cite{Pen}, and could also have a role for the geometry of angular momentum in actual quantum systems. A differential calculus that works with it was, however, only recently proposed\cite[Example~1.46]{BegMa} and it is significant that, using this, we have now solved its quantum Riemannian geometry (found a natural quantum Levi-Civita connection) for any quantum metric, including the canonical  rotationally invariant `round' metric given by $\delta_{ij}$. 

After understanding this moduli of quantum Riemannian geometries on the fuzzy sphere, we went on and constructed Euclidean quantum gravity on it, where we integrate over all quantum geometries. Even though the fuzzy sphere is an infinite-dimensional algebra which becomes functions (in some form) on the usual sphere in the classical limit, its quantum geometry turned out to be much more rigid and to admit only quantum metrics of the form of a single $3\times 3$ matrix $g_{ij}$ transported over the whole algebra. It turned out that the quantum Levi-Civita connection, hence the whole moduli of quantum Riemannian geometry could similarly be developed with constant coefficients, hence the quantum geometry behaves effectively like just one point. This was not put in by hand, but forced by the rigidity of the axioms of noncommutative geometry and the noncommutativity of our particular algebra. Moreover, our results were strikingly similar to quantum gravity on a quadrilateral in \cite{Ma:squ} even though the details are completely different not to mention that that model is Lorentzian with $\imath$ in action, whereas ours is Euclidean. In both cases, the functional integral over all metrics of the natural action built from the quantum Ricci scalar has  UV and IR divergences, appearing in our case at $\lambda_i=0,\infty$. In both cases, the divergence at metric zero modes cancels in the ratio of functional integrals when we look at expectation values. In both cases, the other divergence is controlled by a cutt-off L and in both cases the expectation value of an $n$-th power of the field diverges as $L^n$, with the result that ratios of expectation values can still be defined as $L\to\infty$. 
In both cases, we found in this way a uniform relative uncertainly in the quantisation of the metric components (in our case, we quanitised the metric eigenvalues but one can also think of this as quantising diagonal metrics). 

There are many interesting directions that one could further explore. On the quantum gravity front, one could introduce matter and see how some kind of Einstein equation emerges out of quantum gravity with matter. A first step here would be to better understand the geometric approach to the stress-energy tensor. One could also consider quantum matter fields on curved FLRW cosmologies $\R\times S^2$ where the $S^2$ is now fuzzy, following the spirit of \cite{ArgMa} where $\R\times \Z_n$ is done using quantum geometry on the polygon $\Z_n$ (this work also solves Euclidean quantum gravity on $\Z_n$).  Finally, one could look at quantum geodesics on the fuzzy sphere using the recent formalism in \cite{BegMa5}. 

We have also constructed a natural fuzzy monopole and there are potentially many applications that could be related to that. In physical terms, this could be relevant to any quantum system where classically one has angular momentum at play, for example the effective geometry around a quantised hydrogen atom. In mathematical terms, an important application would be towards the programme of `geometric realisation' - constructing examples of Connes' notion of `spectral triple'\cite{Con} or `axiomatic Dirac operator' but in a geometric manner starting with the quantum differential structure, a spinor bundle with connection, a quantum metric and a Clifford structure. This was done for the $q$-sphere in \cite{BegMa3} and one might try to follow the same steps. Thus, the spinor bundle on the sphere should be of the form $\CS=\CS_+\oplus\CS_-$ where $\CS$ is the charge 1 monopole line bundle as found in Section~\ref{secmon} and $\CS_-$ is its dual. The Clifford structure is a `Clifford action' map $\Omega^1\tens_A\CS\to \CS$ obeying certain axioms \cite{BegMa3,BegMa} of compatibility with the $*$-structure, with the connection on $\CS$ and with the quantum Levi-Civita connection. On the other hand, the Clifford structure for $q$-sphere case was found from the holomorphic structure of its 2D calculus, which does not apply here. This nevertheless merits further study and will be attempted elsewhere. If a geometrically-realised spectral triple can be constructed on the fuzzy sphere, it may (or may not) descend when $\lambda_p=1/n$ to the quotient $c_\lambda[S^2]$ isomorphic to $M_n$, i.e. to the reduced matrix fuzzy spheres. This may then complement (or perhaps relate to) the finite fuzzy Dirac operators constructed in \cite{Bar} using Connes formalism and  starting from  the reduced noncommutative torus.

\end{document}